\documentclass[11pt]{article}
\usepackage{amsmath,amsthm,amssymb,amsfonts,amscd}
\usepackage{amsxtra}
\usepackage{eucal}
\usepackage{mathrsfs}
\usepackage{color}
\usepackage{xcolor}
\usepackage{hyperref}
\usepackage{enumerate}
\usepackage{blkarray}
\usepackage{tikz}
\usepackage{enumitem}
\usepackage{tcolorbox}
\usepackage{algorithm}
\usepackage{algpseudocode}
\usepackage[a4paper,margin=2cm]{geometry}
\definecolor{forestgreen(traditional)}{rgb}{0.0, 0.27, 0.13}
\definecolor{forestgreen(web)}{rgb}{0.13, 0.55, 0.13}
\definecolor{airforceblue}{rgb}{0.36, 0.54, 0.66}

\newcommand\GF{\operatorname{GF}}

\newcommand{\rnk}{\mathrm{rank}}

\hypersetup{
  colorlinks,
  citecolor=forestgreen(web),
  linkcolor=airforceblue,
  urlcolor=magenta,
  pdftitle={Gamma-graphic delta-matroids and their applications},
  pdfauthor={Donggyu Kim, Duksang Lee, and Sang-il Oum}}

\newtheorem{theorem}{Theorem}[section]
\newtheorem{lemma}[theorem]{Lemma}
\newtheorem{conjecture}[theorem]{Conjecture}
\newtheorem{claim}{Claim}
\newtheorem{corollary}[theorem]{Corollary}
\newtheorem{proposition}[theorem]{Proposition}
\newtheorem{case}{Case}

\newtheorem*{theorem1}{Theorem~\ref{thm:delta}}
\newtheorem*{theorem2}{Theorem~\ref{thm:main_algo}}
\newtheorem*{theorem3}{Theorem~\ref{thm:even}}
\newtheorem*{theorem4}{Theorem~\ref{thm:repre1}}
\newtheorem*{theorem5}{Theorem~\ref{thm:repre2}}

\theoremstyle{definition}

\theoremstyle{remark}

\usepackage{authblk}
\usepackage{lineno}
\title{$\Gamma$-graphic delta-matroids and their applications}
\author[2,1]{Donggyu Kim\thanks{Supported by the Institute for Basic Science (IBS-R029-C1).}}
\author[$*$2,1]{Duksang Lee}
\author[$*$1,2]{Sang-il Oum}
\affil[1]{Discrete Mathematics Group,
Institute for Basic Science (IBS),
Daejeon,~South~Korea}
\affil[2]{Department of Mathematical Sciences, KAIST, Daejeon, South~Korea}
\affil[ ]{Email: \texttt{donggyu@kaist.ac.kr}, \texttt{duksang@kaist.ac.kr}, \texttt{sangil@ibs.re.kr}}
\date{\today}	

\begin{document}
\maketitle
\begin{abstract}
For an abelian group $\Gamma$, a $\Gamma$-labelled graph is a graph whose vertices are labelled by elements of $\Gamma$.
We prove that a certain collection of edge sets of a $\Gamma$-labelled graph forms a delta-matroid, which we call a \emph{$\Gamma$-graphic} delta-matroid, and provide a polynomial-time algorithm to solve the separation problem, which allows us to apply the symmetric greedy algorithm of Bouchet to find a maximum weight feasible set in such a delta-matroid.
We present two algorithmic applications on graphs; \textsc{Maximum Weight Packing of Trees of Order Not Divisible by $k$} and \textsc{Maximum Weight $S$-Tree Packing}.
We also discuss various properties of $\Gamma$-graphic delta-matroids.
\end{abstract}

\section{Introduction}\label{sec: intro}
We introduce the class of $\Gamma$-graphic delta-matroids arising from graphs whose vertices are labelled by elements of an abelian group $\Gamma$. This allows us to show that the following problems are solvable in polynomial time by using the symmetric greedy algorithm~\cite{Bouchet1987sym}.
\begin{tcolorbox}
  \textsc{Maximum Weight Packing of Trees of Order Not Divisible by $k$} \\
  \textbf{Input:} An integer $k\geq 2$, a graph $G$, and a weight $w : E(G) \rightarrow \mathbb{Q}$. \\
  \textbf{Problem:} Find vertex-disjoint trees $T_1, T_2, \dots, T_m$ for some $m$ such that $|V(T_i)|\not\equiv 0 \pmod{k}$ for each $i\in\{1,\ldots,m\}$ and $\sum_{i=1}^m \sum_{e \in E(T_i)} w(e)$ is maximized.
\end{tcolorbox}

For a vertex set $S$ of a graph $G$, a subgraph of $G$ is an \emph{$S$-tree} if it is a tree intersecting $S$.

\begin{tcolorbox}
  \textsc{Maximum Weight $S$-Tree Packing} \\
  \textbf{Input:} A graph $G$, a nonempty subset $S$ of $V(G)$, and a weight $w : E(G) \rightarrow \mathbb{Q}$. \\
  \textbf{Problem:} Find vertex-disjoint $S$-trees $T_1, T_2, \dots, T_m$ for some $m$ such that
  $\bigcup_{i=1}^m V(T_i) = V(G)$ and
  $\sum_{i=1}^m \sum_{e \in E(T_i)} w(e)$ is maximized.
\end{tcolorbox}

Let $\Gamma$ be an abelian group. We assume that $\Gamma$ is an additive group. A \emph{$\Gamma$-labelled graph} is a pair $(G,\gamma)$ of a graph $G$ and a map $\gamma : V(G) \rightarrow \Gamma$.
A subgraph $H$ of $G$ is \emph{$\gamma$-nonzero} if, for each component $C$ of $H$, 
\begin{enumerate}[label=(G\arabic*)]
  \item $\sum_{v\in V(C)} \gamma(v) \neq 0$ or $\gamma|_{V(C)}\equiv 0$, and
  \item if $\gamma|_{V(C)}\equiv 0$, then $G[V(C)]$ is a component of $G$.
  \end{enumerate}
A subset $F$ of $E(G)$ is \emph{$\gamma$-nonzero} in $G$ if a subgraph $(V(G),F)$ is $\gamma$-nonzero.
A subset $F$ of $E(G)$ is \emph{acyclic} in $G$ if a subgraph $(V(G),F)$ has no cycle. 

Bouchet~\cite{Bouchet1987sym} introduced delta-matroids which are set systems $(E,\mathcal{F})$ satisfying certain axioms.
Our first theorem proves that the set of acyclic $\gamma$-nonzero sets in a $\Gamma$-labelled graph $(G,\gamma)$ forms a delta-matroid, which we call a \emph{$\Gamma$-graphic} delta-matroid.
For sets $X$ and $Y$, let $X\triangle Y=(X-Y)\cup(Y-X)$.

\begin{theorem}
  \label{thm:delta}
  Let $\Gamma$ be an abelian group and $(G,\gamma)$ be a $\Gamma$-labelled graph. 
  If $\mathcal{F}$ is the set of acyclic $\gamma$-nonzero sets in $G$, then the following hold.
  \begin{enumerate}[label=\rm(\arabic*)]
  \item $\mathcal{F}\neq\emptyset$.
  \item For $X,Y\in\mathcal{F}$ and $e\in X\triangle Y$, there exists $f\in X\triangle Y$ such that $X\triangle\{e,f\}\in\mathcal{F}$.
  \end{enumerate}
\end{theorem}
Bouchet~\cite{Bouchet1987sym} proved that the symmetric greedy algorithm finds a maximum weight set in $\mathcal{F}$ for a delta-matroid $(E,\mathcal{F})$. But it requires the \emph{separation oracle}, which determines, for two disjoint subsets $X$ and $Y$ of $E$, whether there exists a set $F\in\mathcal{F}$ such that $X\subseteq F$ and $F\cap Y=\emptyset$.
We provide the separation oracle that runs in polynomial time for $\Gamma$-graphic delta-matroids given by $\Gamma$-labelled graphs. 
As a consequence, we prove the following theorem.

\begin{tcolorbox}
  \textsc{Maximum Weight Acyclic $\gamma$-nonzero Set} \\
  \textbf{Input:} A $\Gamma$-labelled graph $(G, \gamma)$ and a weight $w : E(G) \rightarrow \mathbb{Q}$. \\
  \textbf{Problem:} Find an acyclic $\gamma$-nonzero set $F$ in $G$ maximizing $\sum_{e \in F} w(e)$.
\end{tcolorbox}

\begin{theorem}\label{thm:main_algo}
  \textsc{Maximum Weight Acyclic $\gamma$-nonzero Set} is solvable in polynomial time.
\end{theorem}

From Theorem~\ref{thm:main_algo}, we can easily deduce that both \textsc{Maximum Weight Packing of Trees of Order Not Divisible by $k$} and \textsc{Maximum Weight $S$-Tree Packing} are solvable in polynomial time.

\begin{corollary}
\label{cor:divk}
\textsc{Maximum Weight Packing of Trees of Order Not Divisible by $k$} is solvable in polynomial time.
\end{corollary}
\begin{proof}
Let $\Gamma = \mathbb{Z}_k$ and $\gamma: V(G) \rightarrow \mathbb{Z}_k$ be a map such that $\gamma(v) = 1$ for each $v\in V(G)$. Then, an edge set $F$ is an acyclic $\gamma$-nonzero set in $(G,\gamma)$ if and only if there exist vertex-disjoint trees $T_{1},\ldots,T_{m}$ for some $m$ such that $\bigcup_{i=1}^{m}E(T_{i})=F$ and $|V(T_{i})|\not\equiv 0\pmod{k}$ for each $i\in\{1,\ldots,m\}$.
\end{proof}

\begin{corollary}
\label{cor:stree}
\textsc{Maximum Weight $S$-Tree Packing} is solvable in polynomial time.
\end{corollary}

\begin{proof}
We may assume that every component of $G$ has a vertex in $S$.
Let $\Gamma = \mathbb{Z}$ and $\gamma: V(G) \rightarrow \mathbb{Z}$ be a map such that
\[
\gamma(v)=
\begin{cases}
1 & \text{if $v\in S$,} \\
0 & \text{otherwise.}
\end{cases}
\] Then, an edge set $F$ is an acyclic $\gamma$-nonzero set in $(G,\gamma)$ if and only if there exist vertex-disjoint $S$-trees $T_{1},\ldots,T_{m}$ for some $m$ such that $\bigcup_{i=1}^{m}V(T_{i})=V(G)$ and $\bigcup_{i=1}^{m}E(T_{i})=F$.
\end{proof}

One of the major motivations to introduce $\Gamma$-graphic delta-matroids is to generalize the concept of graphic delta-matroids introduced by Oum~\cite{Oum2009circle}, which are precisely $\mathbb{Z}_2$-graphic delta-matroids.
Oum~\cite{Oum2009circle} proved that every minor of graphic delta-matroids is graphic.
We will prove that every minor of a $\Gamma$-graphic delta-matroid is $\Gamma$-graphic.

A delta-matroid $(E,\mathcal{F})$ is \emph{even} if $|X\triangle Y|$ is even for all $X, Y\in\mathcal{F}$.
Oum~\cite{Oum2009circle} proved that every graphic delta-matroid is even.
We characterize even $\Gamma$-graphic delta-matroids as follows.

\begin{theorem}
  \label{thm:even}
  Let $\Gamma$ be an abelian group.
  Then a $\Gamma$-graphic delta-matroid is even if and only if it is graphic.
\end{theorem}

Bouchet~\cite{Bouchet1988repre} proved that for a symmetric or skew-symmetric matrix $A$ over a field $\mathbb{F}$, the set of index sets of nonsingular principal submatrices of $A$ forms a delta-matroid, which we call a delta-matroid \emph{representable over $\mathbb{F}$}. 
Oum~\cite{Oum2009circle} proved that every graphic delta-matroid is representable over $\GF(2)$.
Our next theorem partially characterizes a pair of an abelian group $\Gamma$ and a field $\mathbb{F}$ such that every $\Gamma$-graphic delta-matroid is representable over $\mathbb{F}$.

If $\mathbb{F}_1$ is a subfield of a field $\mathbb{F}_2$, then $\mathbb{F}_2$ is an \emph{extension field} of $\mathbb{F}_1$, denoted by $\mathbb{F}_2/\mathbb{F}_1$. The \emph{degree} of a field extension $\mathbb{F}_2/\mathbb{F}_1$, denoted by $[\mathbb{F}_2:\mathbb{F}_1]$, is the dimension of $\mathbb{F}_{2}$ as a vector space over $\mathbb{F}_1$.

\begin{theorem}\label{thm:repre1}
  Let $p$ be a prime, $k$ be a positive integer, and $\mathbb{F}$ be a field of characteristic $p$. If $[\mathbb{F}:\GF(p)]\geq k$, then every $\mathbb{Z}_p^k$-graphic delta-matroid is representable over $\mathbb{F}$.
 \end{theorem}

 For a prime $p$, an abelian group is an \emph{elementary abelian $p$-group} if every nonzero element has order $p$.

\begin{theorem}\label{thm:repre2}
 Let $\mathbb{F}$ be a finite field of characteristic $p$ and $\Gamma$ be an abelian group.
 If every $\Gamma$-graphic delta-matroid is representable over $\mathbb{F}$, then $\Gamma$ is an elementary abelian $p$-group.
\end{theorem}

Theorems~\ref{thm:repre1} and~\ref{thm:repre2} allow us to partially characterize pairs of a finite field $\mathbb{F}$ and an abelian group $\Gamma$ for which every $\Gamma$-graphic delta-matroid is representable over $\mathbb{F}$ as follows.
We omit its easy proof.

\begin{corollary}
\label{cor:char}
Let $\Gamma$ be a finite abelian group of order at least $2$ and $\mathbb{F}$ be a finite field.
\begin{enumerate}[label=\rm(\roman*)]
  \item\label{item:char1} For every prime $p$ and integers $1 \leq k \leq \ell$, every $\mathbb{Z}_p^k$-graphic delta-matroid is representable over $\GF(p^\ell)$.
  \item\label{item:char2} If every $\Gamma$-graphic delta-matroid is representable over $\mathbb{F}$, then $\Gamma$ is isomorphic to $\mathbb{Z}_p^k$ and $\mathbb{F}$ is isomorphic to $\GF(p^{\ell})$ for a prime $p$ and positive integers $k$ and $\ell$.
\end{enumerate}
\end{corollary}

We suspect that the following could be the complete characterization.

\begin{conjecture}
  Let $\Gamma$ be a finite abelian group of order at least $2$ and $\mathbb{F}$ be a finite field.
  Then every $\Gamma$-graphic delta-matroid is representable over $\mathbb{F}$ if and only if $(\Gamma, \mathbb{F}) = (\mathbb{Z}_p^k, \GF(p^\ell))$ for some prime $p$ and positive integers $k \leq \ell$.
\end{conjecture}

This paper is organized as follows. In Section~\ref{sec: preliminaries}, we review some terminologies and results on delta-matroids and graphic delta-matroids. 
In Section~\ref{sec: delta-matroids from group-labelled graphs}, we introduce $\Gamma$-graphic delta-matroids. We show that the class of $\Gamma$-graphic delta-matroids is closed under taking minors in Section~\ref{sec: minors of group-labelled graphs}.
In Section~\ref{sec: applications}, we present a polynomial-time algorithm to solve \textsc{Maximum Weight Acyclic $\gamma$-nonzero Set}, proving Theorem~\ref{thm:main_algo}.
We characterize even $\Gamma$-graphic delta-matroids in Section~\ref{sec: even gamma-graphic delta-matroids}.
In Section~\ref{sec: representations of Gamma-graphic delta-matroids}, we prove Theorems~\ref{thm:repre1} and~\ref{thm:repre2}.

\section{Preliminaries}\label{sec: preliminaries}

In this paper, all graphs are finite and may have parallel edges and loops.
A graph is \emph{simple} if it has neither loops nor parallel edges. For a graph $G$, \emph{contracting} an edge $e$ is an operation to obtain a new graph $G/e$ from $G$ by deleting $e$ and identifying ends of $e$. 
For a set $X$ and a positive integer $s$, let $\binom{X}{s}$ be the set of $s$-element subsets of $X$. For two sets $A$ and $B$, let $A\triangle B=(A-B)\cup(B-A)$. 
For a function $f : X \rightarrow Y$ and a subset $A \subseteq X$, we write $f|_A$ to denote the restriction of $f$ on $A$.

\paragraph{Delta-matroids.} Bouchet~\cite{Bouchet1987sym} introduced delta-matroids.
A \emph{delta-matroid} is a pair $M=(E, \mathcal{F})$ of a finite set $E$ and a nonempty set $\mathcal{F}$ of subsets of $E$ such that if $X,Y\in\mathcal{F}$ and $x\in X\triangle Y$, then there is $y\in X\triangle Y$ such that $X\triangle\{x,y\}\in\mathcal{F}$.
We write $E(M) = E$ to denote the \emph{ground set} of $M$.
An element of $\mathcal{F}$ is called a \emph{feasible} set.
An element of $E$ is a \emph{loop} of $M$ if it is not contained in any feasible set of~$M$.
An element of $E$ is a \emph{coloop} of $M$ if it is contained in every feasible set of $M$.

\paragraph{Minors.}
For a delta-matroid $M=(E, \mathcal{F})$ and a subset $X$ of $E$, we can obtain a new delta-matroid $M\triangle X=(E, \mathcal{F}\triangle X)$ from $M$ where $\mathcal{F}\triangle X=\{F\triangle X : F\in\mathcal{F}\}$.
This operation is called \emph{twisting} a set $X$ in $M$.
A delta-matroid $N$ is \emph{equivalent} to $M$ if $N=M\triangle X$ for some set $X$. 

If there is a feasible subset of $E-X$, then $M\setminus X=(E-X,\mathcal{F}\setminus X)$ is a delta-matroid where $\mathcal{F}\setminus X=\{F\in\mathcal{F} : F\cap X=\emptyset\}$. This operation of obtaining $M\setminus X$ is called the \emph{deletion} of $X$ in $M$.
A delta-matroid $N$ is a \emph{minor} of a delta-matroid $M$ if $N=M\triangle X\setminus Y$ for some subsets $X,Y$ of $E$.

A delta-matroid is \emph{normal} if $\emptyset$ is feasible. A delta-matroid is \emph{even} if $|X\triangle Y|$ is even for all feasible sets $X$ and $Y$.
It is easy to see that all minors of even delta-matroids are even.

The following theorem gives the minimal obstruction for even delta-matroids, which is implied by Bouchet~{\cite[Lemma 5.4]{Bouchet1989map}}.

\begin{theorem}[Bouchet~\cite{Bouchet1989map}]\label{thm: excluded minor for even delta-matroids}
  A delta-matroid is even if and only if it does not have a minor isomorphic to $(\{e\}, \{\emptyset, \{e\}\})$.
\end{theorem}

\begin{lemma}
\label{lem:delandcont}
Let $N$ be a minor of a delta-matroid $M$ such that $|E(M)|>|E(N)|$. Then there exists an element $e\in E(M)- E(N)$ such that $N$ is a minor of $M\setminus e$ or a minor of $M\triangle\{e\}\setminus e$.
\end{lemma}
\begin{proof}
Since $N$ is a minor of $M$ and $|E(M)|>|E(N)|$, there exist $X,Y\subseteq E$ such that $N=M\triangle X\setminus Y$ and $|Y|\geq 1$. So there exists $e\in Y = E(M)- E(N)$. If $e\notin X$, then $N=(M\setminus e)\triangle X\setminus(Y-\{e\})$ and so $N$ is a minor of $M\setminus e$. So we may assume that $e\in X$. Then $N=(M\triangle\{e\}\setminus e)\triangle(X\setminus\{e\})\setminus(Y-\{e\})$ and so $N$ is a minor of $M\triangle\{e\}\setminus e$.
\end{proof}

\paragraph{Representable delta-matroids.}
For an $R\times C$ matrix $A$ and subsets $X$ of $R$ and $Y$ of $C$, we write $A[X,Y]$ to denote the $X \times Y$ submatrix of $A$.
For an $E\times E$ square matrix $A$ and a subset $X$ of $E$, we write $A[X]$ to denote $A[X,X]$, which is called an $X\times X$ \emph{principal} submatrix of $A$.

For an $E\times E$ square matrix $A$, let $\mathcal{F}(A)=\left\{ X\subseteq E : A[X]\text{ is nonsingular} \right\}$. We assume that $A[\emptyset]$ is nonsingular and so $\emptyset\in\mathcal{F}(A)$.
Bouchet~\cite{Bouchet1988repre} proved that, $(E, \mathcal{F}(A))$ is a delta-matroid if $A$ is an $E\times E$ symmetric or skew-symmetric matrix.
A delta-matroid $M=(E,\mathcal{F})$ is \emph{representable over} a field $\mathbb{F}$ if $\mathcal{F}=\mathcal{F}(A)\triangle X$ for a symmetric or skew-symmetric matrix $A$ over $\mathbb{F}$ and a subset $X$ of $E$. Since $\emptyset\in\mathcal{F}(A)$, it is natural to define representable delta-matroids with twisting so that the empty set is not necessarily feasible in representable delta-matroids. 

A delta-matroid is \emph{binary} if it is representable over $\GF(2)$.
Note that all diagonal entries of a skew-symmetric matrix are zero, even if the characteristic of a field is $2$.

\begin{proposition}[Bouchet~\cite{Bouchet1988repre}]
\label{prop: normal repre}
  Let $M=(E, \mathcal{F})$ be a delta-matroid.
  Then $M$ is normal and representable over a field $\mathbb{F}$ if and only if there is an $E \times E$ symmetric or skew-symmetric matrix $A$ over $\mathbb{F}$ such that $\mathcal{F}=\mathcal{F}(A)$.
\end{proposition}

\begin{lemma}[Geelen~{\cite[page 27]{Geelen1996thesis}}]
\label{lem:skew}
Let $M$ be a delta-matroid representable over a field $\mathbb{F}$. Then $M$ is even if and only if $M$ is representable by a skew-symmetric matrix over $\mathbb{F}$.
\end{lemma}

\paragraph{Pivoting.} For a finite set $E$ and a symmetric or skew-symmetric $E \times E$ matrix $A$, if $A$ is represented by
\[
A=
\begin{blockarray}{ccc}
  & X & Y \\
 \begin{block}{c(cc)}
 X& \alpha &\beta \\
 Y & \gamma & \delta \\
\end{block}
\end{blockarray}
\]
after selecting a linear ordering of $E$
and $A[X]=\alpha$  is nonsingular, then let
\[
A*X=
\begin{blockarray}{ccc}
  & X & Y \\
 \begin{block}{c(cc)}
 X& \alpha^{-1} &\alpha^{-1}\beta \\
 Y & -\gamma \alpha^{-1} & \delta-\gamma\alpha^{-1}\beta \\
\end{block}
\end{blockarray}
\]
This operation is called \emph{pivoting}. Tucker~\cite{Tucker1960matrix} proved that when $A[X]$ is nonsingular, $A*X[Y]$ is nonsingular if and only if $A[X\triangle Y]$ is nonsingular for each subset $Y$ of $E$.
Hence, if $X$ is a feasible set of a delta-matroid $M=(E, \mathcal{F}(A))$, then $M\triangle X=(E, \mathcal{F}(A*X))$.
It implies that all minors of delta-matroids representable over a field $\mathbb{F}$ are representable over $\mathbb{F}$~\cite{Bouchet1991}.

\paragraph{Greedy algorithm.}
Let $M = (E, \mathcal{F})$ be a set system such that $E$ is finite and $\mathcal{F}\neq\emptyset$.
A pair $(X,Y)$ of disjoint subsets $X$ and $Y$ of $E$ is \emph{separable} in~$M$ if there exists a set $F\in\mathcal{F}$ such that $X \subseteq F$ and $Y \cap F = \emptyset$.
The following theorem characterizes delta-matroids in terms of a greedy algorithm. 
Note that this greedy algorithm requires an oracle which answers whether a pair $(X,Y)$ of disjoint subsets $X$ and $Y$ of $E$ is separable in $M$.
\begin{theorem}[Bouchet~\cite{Bouchet1987sym}; see Moffatt~\cite{Moffatt2019delta}]
  \label{thm:greedy_algo}
  Let $M = (E, \mathcal{F})$ be a set system such that $E$ is finite and $\mathcal{F}\neq\emptyset$.
  Then $M$ is a delta-matroid if and only if the symmetric greedy algorithm in Algorithm~\ref{algo:greedy_algo} gives a set $F \in \mathcal{F}$ maximizing $\sum_{e\in F} w(e)$ for each $w : E \rightarrow \mathbb{R}$. 
\end{theorem}

\begin{algorithm}
  \caption{Symmetric greedy algorithm}\label{algo:greedy_algo}
  \begin{algorithmic}[1]
    \Function{Symmetric Greedy Algorithm}{$M, w$}
    \Comment{$M=(E,\mathcal{F})$ and $w: E \rightarrow \mathbb{R}$}
      \State Enumerate $E = \{e_1, e_2, \dots, e_n\}$ such that $|w(e_1)| \geq |w(e_2)| \geq \dots \geq |w(e_n)|$
      \State $X \gets \emptyset$ and $Y \gets \emptyset$
      \For{$i \gets 1$ \textbf{to} $n$}
        \If{$w(e_i) \geq 0$}
          \If{$(X \cup \{e_i\}, Y$) is separable}
            \State $X \gets X \cup \{e_i\}$
          \Else
            \State $Y \gets Y \cup \{e_i\}$
          \EndIf
        \Else
          \If{$(X, Y \cup \{e_i\})$ is separable}
            \State $Y \gets Y \cup \{e_i\}$
          \Else
            \State $X \gets X \cup \{e_i\}$
      \EndIf
        \EndIf
      \EndFor
    \EndFunction
    \State \Return{X} \Comment{$X \in \mathcal{F}$}
  \end{algorithmic}
\end{algorithm}

\paragraph{Graphic delta-matroids.}
Oum~\cite{Oum2009circle} introduced graphic delta-matroid.
A \emph{graft} is a pair $(G,T)$ of a graph $G$ and a subset $T$ of $V(G)$.
A subgraph $H$ of $G$ is \emph{$T$-spanning} in $G$ if $V(H) = V(G)$, for each component $C$ of $H$, either
\begin{enumerate}[label=(\roman*)]
  \item $|V(C) \cap T|$ is odd, or
  \item $V(C) \cap T = \emptyset$ and $G[V(C)]$ is a component of $G$.
\end{enumerate}
An edge set $F$ of $G$ is \emph{$T$-spanning} in $G$ if a subgraph $(V(G), F)$ is $T$-spanning in $G$.
For a graft $(G,T)$, let $\mathcal{G}(G,T) = (E(G), \mathcal{F})$ where $\mathcal{F}$ is the set of acyclic $T$-spanning sets in~$G$.
Oum~\cite{Oum2009circle} proved that $\mathcal{G}(G,T)$ is an even binary delta-matroid.
A delta-matroid is \emph{graphic} if it is equivalent to $\mathcal{G}(G,T)$ for a graft $(G,T)$.

\section{Delta-matroids from group-labelled graphs}\label{sec: delta-matroids from group-labelled graphs}
Let $\Gamma$ be an abelian group. 
A \emph{$\Gamma$-labelled graph} $(G,\gamma)$ is a pair of a graph $G$ and a map $\gamma : V(G)\rightarrow\Gamma$.
We say $\gamma \equiv 0$ if $\gamma(v) = 0$ for all $v \in V(G)$.
A $\Gamma$-labelled graph $(G, \gamma)$ and a $\Gamma'$-labelled graph $(G', \gamma')$ are \emph{isomorphic} if there are a graph isomorphism $f$ from $G$ to $G'$ and a group isomorphism $\phi: \Gamma \rightarrow \Gamma'$ such that $\phi ( \gamma (v)) = \gamma'(f(v))$ for each $v \in V(G)$.

A subgraph $H$ of $G$ is \emph{$\gamma$-nonzero} if, for each component $C$ of $H$, 
\begin{enumerate}[label=\rm(G\arabic*)]
\item\label{item:def1} $\sum_{v\in V(C)} \gamma(v) \neq 0$ or $\gamma|_{V(C)}\equiv 0$, and
\item\label{item:def2} if $\gamma|_{V(C)} \equiv 0$, then $G[V(C)]$ is a component of $G$.
\end{enumerate}
An edge set $F$ of $E(G)$ is \emph{$\gamma$-nonzero} in $G$ if a subgraph $(V(G),F)$ is $\gamma$-nonzero. An edge set $F$ of $E(G)$ is \emph{acyclic} in $G$ if a subgraph $(V(G),F)$ has no cycle.

For an abelian group $\Gamma$ and a $\Gamma$-labelled graph $(G,\gamma)$, let $\mathcal{F}$ be the set of acyclic $\gamma$-nonzero sets in~$G$.
Now we are ready to show Theorem~\ref{thm:delta}, which proves that $(E(G),\mathcal{F})$ is a delta-matroid.
We denote $(E(G),\mathcal{F})$ by $\mathcal{G}(G,\gamma)$.
A delta-matroid $M$ is \emph{$\Gamma$-graphic} if there exist a $\Gamma$-labelled graph $(G, \gamma)$ and $X\subseteq E(G)$ such that $M=\mathcal{G}(G, \gamma)\triangle X$.

\begin{theorem1}
  Let $\Gamma$ be an abelian group and $(G,\gamma)$ be a $\Gamma$-labelled graph. 
  If $\mathcal{F}$ is the set of acyclic $\gamma$-nonzero sets in $G$, then the following hold.
  \begin{enumerate}[label=\rm(\arabic*)]
  \item\label{item:d1} $\mathcal{F}\neq\emptyset$.
  \item\label{item:d2} For $X,Y\in\mathcal{F}$ and $e\in X\triangle Y$, there exists $f\in X\triangle Y$ such that $X\triangle\{e,f\}\in\mathcal{F}$.
  \end{enumerate}
\end{theorem1}

\begin{proof}%
By considering each component, we may assume that $G$ is connected.
If $\gamma \equiv 0$, then we choose a vertex $v$ of $G$ and a map $\gamma' : V(G) \rightarrow \Gamma$ such that $\gamma'(u) \neq 0$ if and only if $u = v$.
Then the set of acyclic $\gamma$-nonzero sets in $G$ is equal to the set of acyclic $\gamma'$-nonzero sets in $G$.
Hence, we can assume that $\gamma$ is not identically zero.
Therefore, a subgraph $H$ of $G$ is $\gamma$-nonzero if and only if $\sum_{u\in V(C)} \gamma(u) \neq 0$ for each component $C$ of $H$.

Let us first prove~\ref{item:d1}, stating that $\mathcal{F} \neq \emptyset$.
Let $S=\{v\in V(G):\gamma(v)\neq 0\}$ and $T$ be a spanning tree of $G$.
Then by the assumption, we have $S\neq\emptyset$.
We may assume that $\sum_{u\in V(G)}\gamma(u)=0$ because otherwise $E(T)$ is acyclic $\gamma$-nonzero in $G$.
Let $e$ be an edge of $T$ such that one of two components $C_1$ and $C_2$ of $T\setminus e$ has exactly one vertex in $S$.
Then $\sum_{u\in V(C_1)} \gamma(u) = - \sum_{u \in V(C_2)} \gamma(u) \neq 0$.
So $E(T) - \{e\}$ is acyclic $\gamma$-nonzero in $G$, and \ref{item:d1} holds.

Now let us prove~\ref{item:d2}.
We proceed by induction on $|E(G)|$.
It is obvious if $|E(G)|=0$.
If there is an edge $g = vw$ in $X\cap Y$, then let $\gamma' : V(G / g) \rightarrow \Gamma$ such that, for each vertex $x$ of $G/g$,
\[
  \gamma'(x)=
  \begin{cases}
    \gamma(v)+\gamma(w) & \text{if $x$ is the vertex of $G/g$ corresponding to $g$,} \\
    \gamma(x)  & \text{otherwise.}
  \end{cases}
\] 
Then both $X - \{g\}$ and $Y - \{g\}$ are acyclic $\gamma'$-nonzero sets in $G/g$.
Let $e \in (X - \{g\}) \triangle (Y - \{g\})=X\triangle Y$.
By the induction hypothesis, there exists $f \in X \triangle Y$ such that $(X - \{g\})\triangle\{e,f\}$ is an acyclic $\gamma'$-nonzero set in $G/g$.

We now claim that $X \triangle \{e,f\}$ is an acyclic $\gamma$-nonzero set in $G$.
It is obvious that $X \triangle \{e,f\}$ is acyclic in $G$.
If $\gamma' \equiv 0$, then $\gamma(v) = -\gamma(w) \neq 0$ and $\gamma(u) = 0$ for every $u$ in $V(G) - \{v,w\}$.
Then $X$ is not $\gamma$-nonzero, contradicting our assumption.
Hence, $\gamma' \not\equiv 0$ and let $C$ be a component of $(V(G), X \triangle \{e,f\})$.
If $C$ contains~$g$, then $\sum_{u \in V(C)} \gamma(u) = \sum_{u \in V(C/g)} \gamma'(u) \neq 0$.
If $C$ does not contain~$g$, then $\sum_{u \in V(C)} \gamma(u) = \sum_{u \in V(C)} \gamma'(u) \neq 0$.
It implies that $X \triangle \{e,f\}$ is $\gamma$-nonzero in $G$, so the claim is verified.

Therefore we may assume that $X \cap Y = \emptyset$.
Let $H_1 = (V(G), X)$ and $H_2 = (V(G), Y)$.
\begin{case}
$e\in X$.
\end{case}
Let $C$ be the component of $H_{1}$ containing $e$ and $C_{1}$, $C_{2}$ be two components of $C\setminus e$.
  If both $\sum_{u\in V(C_{1})}\gamma(u)$ and $\sum_{u\in V(C_{2})}\gamma(u)$ are nonzero, then $X\triangle\{e\}$ is acyclic $\gamma$-nonzero and so we can choose $f=e$.
 So we may assume that $\sum_{u\in V(C_{1})}\gamma(u)=0$ and therefore
\[
\sum_{u\in V(C_{2})}\gamma(u)=\sum_{u\in V(C)}\gamma(u)-\sum_{u\in V(C_{1})}\gamma(u)\neq 0.
\]

  If there exists $f\in Y$ joining a vertex in $V(C_{1})$ to a vertex in $V(G)-V(C_{1})$, then $X\triangle\{e,f\}$ is acyclic $\gamma$-nonzero.
 Therefore, we may assume that there is a component $D_{1}$ of $H_{2}$ such that $V(D_{1})\subseteq V(C_{1})$. 
Since $\sum_{u\in V(D_{1})}\gamma(u)\neq 0$, there is a vertex $x$ of $D_{1}$ such that $\gamma(x)\neq 0$. So $\gamma|_{V(C_{1})}\not\equiv 0$ and there is an edge $f$ of $C_{1}$ such that one of the components of $C_{1}\setminus f$, say $U$, has exactly one vertex $v$ with $\gamma(v)\neq 0$. If $U'$ is the component of $C_{1}\setminus f$ other than $U$, then $\sum_{u\in V(U')}\gamma(u)=-\sum_{u\in V(U)}\gamma(u)\neq 0$.
So $X\triangle\{e,f\}$ is acyclic $\gamma$-nonzero.
\begin{case}
$e\in Y$.
\end{case}
Let $\tilde{H}=(V(G), X\cup\{e\})$.
  If $\tilde{H}$ contains a cycle $D$, then, since $X$ and $Y$ are acyclic, $D$ is a unique cycle of $\tilde{H}$ and there is an edge $f\in E(D)-Y$. Then $X\triangle\{e,f\}$ is acyclic $\gamma$-nonzero.
Therefore, we can assume that $e$ joins two distinct components $C'$, $C''$ of $H_{1}$.

Since $\sum_{u\in V(C')} \gamma(u) \neq 0$, there is an edge $f$ of $C'$ such that one of the components of $C'\setminus f$, say $U$, has exactly one vertex $v$ with $\gamma(v)\neq 0$. If $U'$ is the component of $C'\setminus f$ other than $U$, then $\sum_{u\in V(U')}\gamma(u)=-\sum_{u\in V(U)}\gamma(u)\neq 0$.
So $X\triangle\{e,f\}$ is acyclic $\gamma$-nonzero.
\end{proof}

\section{Minors of group-labelled graphs}\label{sec: minors of group-labelled graphs}
\label{sec:minor}
Let $\Gamma$ be an abelian group.
Now we define minors of $\Gamma$-labelled graphs as follows. Let $(G, \gamma)$ be a $\Gamma$-labelled graph and $e=uv$ be an edge of $G$.
Then $(G, \gamma)\setminus e = (G\setminus e, \gamma)$ is the $\Gamma$-labelled graph obtained by \emph{deleting} the edge $e$ from $(G,\gamma)$.
For an isolated vertex $v$ of $G$, $(G, \gamma)\setminus v = (G\setminus v, \gamma|_{V(G)-\{v\}})$ is the $\Gamma$-labelled graph obtained by \emph{deleting} the vertex $v$ from $(G,\gamma)$.
If $e$ is not a loop, then let $(G, \gamma)/e=(G/e, \gamma')$ such that, for each $x\in V(G/e)$,
\[
  \gamma'(x)=
  \begin{cases}
    \gamma(u)+\gamma(v) & \text{if $x$ is the vertex of $G/e$ corresponding to $e$,} \\
    \gamma(x)  & \text{otherwise.}
  \end{cases}
\] If $e$ is a loop, then let $(G, \gamma) / e = (G, \gamma) \setminus e$.
\emph{Contracting} the edge $e$ is an operation obtaining $(G,\gamma)/e$ from $(G,\gamma)$. 
For an edge set $X = \{e_1,\dots,e_t\}$, let $(G,\gamma)/X = (G,\gamma)/e_1/\dots/e_t$ and $(G,\gamma)\setminus X=(G\setminus X,\gamma)$.
A $\Gamma$-labelled graph $(G', \gamma')$ is a \emph{minor} of $(G, \gamma)$ if $(G', \gamma')$ is obtained from $(G, \gamma)$ by deleting some edges, contracting some edges, and deleting some isolated vertices. Let $\kappa(G, \gamma)$ be the number of components $C$ of $G$ such that $\gamma(x) = 0$ for all $x \in V(C)$. An edge $e$ of $G$ is a \emph{$\gamma$-bridge} if $\kappa((G, \gamma) \setminus e) > \kappa(G, \gamma)$.
A non-loop edge $e=uv$ of $G$ is a \emph{$\gamma$-tunnel} if, for the component $C$ of $G$ containing $e$, the following hold:
\begin{enumerate}
  \item[(i)] For each $x\in V(C)$, $\gamma(x)\neq 0$ if and only if $x\in\{u,v\}$.
  \item[(ii)] $\gamma(u)+\gamma(v)=0$.
\end{enumerate}
From the definition of a $\gamma$-tunnel, it is easy to see that an edge $e$ is a $\gamma$-tunnel in $G$ if and only if $\kappa((G,\gamma)/e) > \kappa(G,\gamma)$.

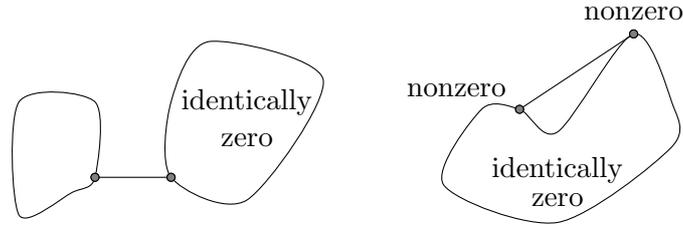
\begin{figure}
  \begin{center}    
    \begin{tikzpicture}
      \tikzstyle{v}=[circle,draw,fill=black!50,inner sep=0pt,minimum width=3pt]
      \draw (0,0) node  [v](x){} -- (1,0) node[v] (y){};
      \draw plot [smooth cycle] coordinates { (x) (0,1) (-1,1) (-1,-.5) (-.3,-.2)};
      \draw plot [smooth cycle] coordinates { (y) (1,1) (1.5,1.8) (3,1.3) (2,-.3)};
      \draw (2,1) node {identically};
      \draw (2,0.5) node {zero};
      \draw (x) node [v]{};
      \draw (y) node [v]{};
    \end{tikzpicture}
    \qquad
    \begin{tikzpicture}
      \tikzstyle{v}=[circle,draw,fill=black!50,inner sep=0pt,minimum width=3pt]
      \draw (0,0) node  [v,label=above left:nonzero](x){} -- (1.5,1) node[v,label=nonzero] (y){};
      \draw plot [smooth cycle] coordinates { (x) (-.5,0) (-1,-1) (.5,-1.5) (2,-.5) (2,.1) (y) (0.5,-.3)};
      \draw (0.5,-.8) node {identically};
      \draw (.5,-1.2) node {zero};
      \draw (x) node [v]{};
      \draw (y) node [v]{};
    \end{tikzpicture}
  \end{center}
  \caption{A $\gamma$-bridge and a $\gamma$-tunnel.}
\end{figure}

The following lemmas are analogous to properties of graphic delta-matroids in Oum~{\cite[Propositions~8, 9, 10, and 11]{Oum2009circle}}.

\begin{lemma}
\label{lem:gbridge}
Let $(G,\gamma)$ be a $\Gamma$-labelled graph and $e$ be an edge of $G$.
The following are equivalent.
\begin{enumerate}[label=\rm{(\roman*)}]
  \item\label{item:b2} Every acyclic $\gamma$-nonzero set in $G$ contains $e$.
  \item\label{item:b1} The edge $e$ is a $\gamma$-bridge in $G$.
  \item\label{item:b3} Every $\gamma$-nonzero set in $G$ contains $e$.
\end{enumerate}
\end{lemma}
\begin{proof}
We may assume that $G$ is connected.
If $\gamma \equiv 0$, then we choose a vertex $v$ of $G$ and take a map $\gamma' : V(G) \rightarrow \Gamma$ such that $\gamma'(v) \neq 0$ and $\gamma'(u) = 0$ for all $u \neq v$.
Then an edge set of $G$ is $\gamma$-nonzero if and only if it is $\gamma'$-nonzero, and $e$ is a $\gamma$-bridge if and only if it is a $\gamma'$-bridge.
So we can assume that $\gamma \not\equiv 0$.
Therefore, an edge set $F$ of $G$ is $\gamma$-nonzero in $G$ if and only if $\sum_{u\in V(C)} \gamma(u) \neq 0$ for each component $C$ of a subgraph $(V(G),F)$.
It is obvious that \ref{item:b3} implies \ref{item:b2}.

We first prove that \ref{item:b2} implies \ref{item:b1}. Suppose that $e$ is not a $\gamma$-bridge.
By~\ref{item:d1} of Theorem~\ref{thm:delta}, $G\setminus e$ has an acyclic $\gamma$-nonzero set $F$.
If $G \setminus e$ is connected, then $F$ is acyclic $\gamma$-nonzero in~$G$.
So we may assume that $G \setminus e$ has exactly two components $C_1$ and $C_2$.
Since $e$ is not a $\gamma$-bridge in $G$, we have $\gamma|_{V(C_1)}$, $\gamma|_{V(C_2)} \not\equiv 0$.
Hence, $\sum_{u\in V(D)} \gamma(u) \neq 0$ for every component $D$ of $(V(G),F)$ by~\ref{item:def1}.
Therefore, $F$ is acyclic $\gamma$-nonzero in $G$ not containing $e$. 

Now let us prove that \ref{item:b1} implies \ref{item:b3}. 
Let $e=uv$ be a $\gamma$-bridge. Then $G\setminus e$ contains a component $C$ such that $\gamma|_{V(C)}\equiv 0$. We may assume that $u\in V(C)$ and $v\notin V(C)$.
Suppose that $G$ has a $\gamma$-nonzero set $F$ which does not contain $e$.
Let $D$ be a component of $(V(G), F)$ containing $u$.
Then $V(D)\subseteq V(C)$ and so $\sum_{u\in V(D)}\gamma(u) = 0$, contradicting that $F$ is $\gamma$-nonzero in $G$.
Therefore, every $\gamma$-nonzero set in $G$ contains the edge $e$.
\end{proof}

\begin{lemma}
\label{lem:deletion}
Let $(G,\gamma)$ be a $\Gamma$-labelled graph. Then, for an edge $e$ of $G$, 
\[
\mathcal{G}((G, \gamma)\setminus e)=
\begin{cases}
\mathcal{G}(G, \gamma)\setminus e & \text{if $e$ is not a $\gamma$-bridge,} \\
\mathcal{G}(G, \gamma)\triangle\{e\}\setminus e & \text{otherwise.}
\end{cases}
\]
\end{lemma}
\begin{proof}
We may assume that $G$ is connected.
If $\gamma \equiv 0$, then we choose a vertex $v$ of $G$ and take a map $\gamma' : V(G) \rightarrow \Gamma$ such that $\gamma'(v) \neq 0$ and $\gamma'(u) = 0$ for all $u \neq v$.
Then an edge set of $G$ is $\gamma$-nonzero if and only if it is $\gamma'$-nonzero, and $e$ is a $\gamma$-bridge if and only if it is a $\gamma'$-bridge.
So we can assume that $\gamma \not\equiv 0$.
Therefore, an edge set $F$ of $G$ is $\gamma$-nonzero in $G$ if and only if $\sum_{u\in V(C)} \gamma(u) \neq 0$ for each component $C$ of a subgraph $(V(G),F)$.

We first consider the case that $e$ is not a $\gamma$-bridge.
Then $\gamma|_{V(C)} \not\equiv 0$ for each component $C$ of $G \setminus e$.
So an edge set $F$ of $G\setminus e$ is $\gamma$-nonzero in $G\setminus e$ if and only if $\sum_{u\in V(D)} \gamma(u) \neq 0$ for each component $D$ of a subgraph $(V(G),F)$.
Therefore, $\mathcal{G}((G, \gamma)\setminus e) = \mathcal{G}(G, \gamma)\setminus e$.

So it is enough to consider the case that $e$ is a $\gamma$-bridge.
Since $\gamma \not\equiv 0$ and $e$ is a $\gamma$-bridge, $G \setminus e$ consists of two components $C_1$ and $C_2$ such that $\gamma|_{V(C_1)} \equiv 0$ and $\gamma|_{V(C_2)} \not\equiv 0$.
Let $v_1$ and $v_2$ be the ends of $e$ such that $v_i \in V(C_i)$ for $i \in \{1,2\}$.

Let $F$ be a feasible set of $\mathcal{G}(G, \gamma)$, which means that $F$ is acyclic $\gamma$-nonzero in $G$.
By Lemma~\ref{lem:gbridge}, $F$ contains $e$.
Let $D$ be the component of $(V(G),F)$ containing $e$ and $D_i$ be the component of $D\setminus e$ containing $v_i$ for $i\in \{1,2\}$.
Since $V(D_1) \subseteq V(C_1)$ and $\gamma|_{V(C_1)} \equiv 0$, we have $V(D_1) = V(C_1)$ and $\sum_{u \in V(D_2)} \gamma(u) = \sum_{u \in V(D)} \gamma(u) - \sum_{u \in V(D_1)} \gamma(u) \neq 0$.
Therefore, $F - \{e\}$ is an acyclic $\gamma$-nonzero set in $G \setminus e$, which implies that $F-\{e\} = F \triangle \{e\}$ is a feasible set of $\mathcal{G}((G, \gamma)\setminus e)$.

Conversely, let $F$ be a feasible set of $\mathcal{G}((G, \gamma)\setminus e)$.
Then $F$ is an acyclic $\gamma$-nonzero set in $G\setminus e$.
We claim that $F\cup\{e\}$ is an acyclic $\gamma$-nonzero set in $G$.
Let $D$ be the component of $(V(G),F\cup\{e\})$ containing $e$ and $D_i$ be the component of $D \setminus e$ containing $v_i$ for $i\in\{1,2\}$.
Then $V(D_i) \subseteq V(C_i)$.
So $\gamma|_{V(D_1)} \equiv 0$ and, since $F$ is acyclic $\gamma$-nonzero in $G\setminus e$, we have $\sum_{u \in V(D_2)} \gamma(u) \neq 0$.
Hence $\sum_{u \in V(D)} \gamma(u) = \sum_{u \in V(D_1)} \gamma(u) + \sum_{u \in V(D_2)} \gamma(u) \neq 0$.
Therefore, $F\cup\{e\} = F \triangle \{e\}$ is an acyclic $\gamma$-nonzero set in $G$, which implies that $F$ is a feasible set of $\mathcal{G}(G, \gamma)\triangle\{e\}\setminus e$. 
\end{proof}

\begin{lemma}
  \label{lem:gtunnel}
  Let $(G,\gamma)$ be a $\Gamma$-labelled graph and $e$ be a non-loop edge of $G$.
  Then the following are equivalent.
  \begin{enumerate}[label=\rm(\roman*)]
    \item\label{item:t1} No acyclic $\gamma$-nonzero set in $G$ contains $e$.
    \item\label{item:t2} The edge $e$ is a $\gamma$-tunnel in $G$.
    \item\label{item:t3} No $\gamma$-nonzero set in $G$ contains $e$.
  \end{enumerate}
\end{lemma}
\begin{proof}
  It is obvious that \ref{item:t3} implies \ref{item:t1}.
  We first show that \ref{item:t1} implies \ref{item:t2}. We may assume that $G$ is connected.
  Let $H$ be a spanning tree containing $e$.
  Then we may assume that $\gamma\not\equiv 0$ and $\sum_{u\in V(H)}\gamma(u)=0$ because otherwise $E(H)$ is acyclic $\gamma$-nonzero in $G$.
  Let $S=\{v\in V(G) :\gamma(v)\neq 0\}$.
  Then $|S| \geq 2$. 
  If $S$ contains a vertex not in $\{x,y\}$, then let $x'$ be a vertex in $S-\{x,y\}$ maximizing $d_{H}(x,x')$. Let $f$ be an edge incident with $x'$ on the path from $x$ to $x'$. Then $H\setminus f$ has components $D_{1}$, $D_{2}$ such that $x'\in V(D_{1})$ and $x\in V(D_{2})$. Then $\sum_{u\in V(D_{1})}\gamma(u)=\gamma(x')\neq 0$ by the choice of $x'$ and $\sum_{u\in V(D_{2})}\gamma(u)=\sum_{u\in V(H)}\gamma(u)-\sum_{u\in V(D_{1})}\gamma(u)=-\gamma(x')\neq 0$. Hence $E(H) - \{f\}$ is an acyclic $\gamma$-nonzero set in $G$ containing $e$.
Therefore, $S=\{x,y\}$ and $e$ is a $\gamma$-tunnel in $G$.

  Now let us prove that \ref{item:t2} implies \ref{item:t3}.
  Let $e=xy$ be a $\gamma$-tunnel of $G$ and $C$ be a component of~$G$ containing $e$.
  Suppose that $G$ has a $\gamma$-nonzero set $F$ containing $e$ and let $D$ be the component of $(V(G), F)$ containing $e$.
  Since $V(D)\subseteq V(C)$ and $e$ is a $\gamma$-tunnel, we have that $\gamma|_{V(D)}\not\equiv 0$ and $\sum_{u\in V(D)}\gamma(u)=\gamma(x)+\gamma(y)=0$, contradicting~\ref{item:def1}.
  Hence $G$ has no $\gamma$-nonzero set containing $e$.
  \end{proof}

\begin{lemma}
\label{lem:contraction}
Let $(G, \gamma)$ be a $\Gamma$-labelled graph. Then, for an edge $e$ of $G$, 
\[
  \mathcal{G}((G, \gamma)/ e)=
  \begin{cases}
    \mathcal{G}(G, \gamma)\triangle\{e\}\setminus e & \text{if $e$ is neither a loop nor a $\gamma$-tunnel,} \\
    \mathcal{G}(G, \gamma)\setminus e & \text{otherwise.}
  \end{cases}
\]
\end{lemma}
\begin{proof}
Let $e^{*}$ be the vertex of $G/e$ corresponding to $e$, and let $\gamma^\ast: V(G/e) \rightarrow \Gamma$ be a map such that $(G/e, \gamma^\ast) = (G, \gamma)/e$.
We first consider the case that $e$ is neither a $\gamma$-tunnel nor a loop.
Let $F$ be a feasible set of $\mathcal{G}((G, \gamma)/e)$, which implies that $F$ is acyclic $\gamma^{*}$-nonzero in $G/e$.
We aim to prove that $F\cup\{e\}$ is acyclic $\gamma$-nonzero in $G$.

Let $C$ be the component of $(V(G),F\cup\{e\})$ containing $e$ and $C^*=C/e$.
Then $C^\ast$ is a component of $(V(G/e), F)$.
If $\sum_{u\in V(C^{*})}\gamma^{*}(u)\neq 0$, then $\sum_{u\in V(C)}\gamma(u)=\sum_{u\in V(C^{*})}\gamma^{*}(u)\neq 0$ and therefore $F\cup\{e\}$ is acyclic $\gamma$-nonzero in~$G$ because $e$ is not a loop.
So we may assume that $\gamma^{*}|_{V(C^{*})}\equiv 0$ and $(G/e)[V(C^{*})]$ is a component of $G/e$.
Then $G[V(C)]$ is a component of $G$ and $\gamma|_{V(C)}\equiv 0$ because $e$ is not a $\gamma$-tunnel.
Since $e$ is not a loop, $F\cup \{e\}$ is acyclic $\gamma$-nonzero in $G$, which means that $F\cup\{e\}=F\triangle\{e\}$ is a feasible set of $\mathcal{G}(G, \gamma)$. Hence $F$ is a feasible set of $\mathcal{G}(G, \gamma)\triangle\{e\}\setminus e$.

Conversely, let $F$ be a feasible set of $\mathcal{G}(G, \gamma)\triangle\{e\}\setminus e$, meaning that $F\cup\{e\}$ is acyclic $\gamma$-nonzero in $G$.
Trivially, $F$ is acyclic in $G/e$.
We claim that $F$ is $\gamma^{*}$-nonzero in $G/e$.
Let $C^\ast$ be the component of $(V(G/e), F)$ containing $e^\ast$ and let $C$ be the component of $(V(G), F \cup \{e\})$ containing $e$.
Then $C^\ast = C / e$.
If $\sum_{u\in V(C)}\gamma(u)\neq 0$, then $\sum_{u\in V(C^{*})}\gamma^{*}(u)=\sum_{u\in V(C)}\gamma(u)\neq 0$ and therefore $F$ is $\gamma^{*}$-nonzero in $G/e$.
So we may assume that $\gamma|_{V(C)}\equiv 0$ and $G[V(C)]$ is a component of $G$.
Then we know that $\gamma^{*}|_{V(C^{*})}\equiv 0$ and $(G/e)[V(C^{*})]$ is a component of $G/e$.
Hence $F$ is $\gamma^{*}$-nonzero in $G/e$.
This proves that $\mathcal{G}((G, \gamma)/e)=\mathcal{G}(G, \gamma)\triangle\{e\}\setminus e$ if $e$ is neither a loop nor a $\gamma$-tunnel.

If $e$ is a loop, then by Lemma~\ref{lem:deletion}, $\mathcal{G}(G, \gamma)\setminus e=\mathcal{G}((G, \gamma)\setminus e)=\mathcal{G}((G, \gamma)/e)$. 

Now we may assume that $e=xy$ is a $\gamma$-tunnel in $G$.
First, let us show that every feasible set $F$ of $\mathcal{G}((G,\gamma)/e)$ is feasible in $\mathcal{G}(G,\gamma)\setminus e$.
Let $C^{*}$ be the component of $(V(G/e),F)$ containing $e^{*}$. Since $e$ is a $\gamma$-tunnel in $G$, we have $\gamma^{*}|_{V(C^{*})}\equiv 0$. Hence, $(G/e)[V(C^{*})]$ is a component of $G/e$ because $F$ is acyclic $\gamma^*$-nonzero in $G/e$. Let $C$ be the component of $(V(G),F\cup\{e\})$ containing $e$.
Then $C/e=C^{*}$ and $C\setminus e$ has two components $C_{1}$, $C_{2}$ such that $x\in V(C_{1})$ and $y\in V(C_{2})$. Observe that $\sum_{u\in V(C_{1})}\gamma(u)=\gamma(x)\neq 0$ and $\sum_{u\in V(C_{2})}\gamma(u)=\gamma(y)\neq 0$. Hence $F$ is acyclic $\gamma$-nonzero in $G$ and $F$ is a feasible set of $\mathcal{G}(G,\gamma)\setminus e$.

Conversely, we want to show that every feasible set $F$ of $\mathcal{G}(G, \gamma)\setminus e$ is feasible in $\mathcal{G}((G,\gamma)/e)$.
Observe that $F$ is acyclic $\gamma$-nonzero in $G$ not containing $e$. Let $C_{1}$ and $C_{2}$ be components of a subgraph $(V(G),F)$ containing $x$ and $y$, respectively. Since $e$ is a $\gamma$-tunnel, we know that $\sum_{u\in V(C_{1})}\gamma(u)=\gamma(x)\neq 0$, $\sum_{u\in V(C_{2})}\gamma(u)=\gamma(y)\neq 0$, and $G[V(C_{1} )\cup V(C_{2})]$ is a component of $G$. So let $C^{*}$ be a component of $(V(G/e),F)$ containing $e^{*}$. Since $\gamma^{*}(e^{*})=\gamma(x)+\gamma(y)=0$ and $G[V(C_{1} )\cup V(C_{2})]$ is a component of $G$, we have $\gamma^{*}|_{V(C^{*})}\equiv 0$ and $(G/e)[V(C^{*})]$ is a component of $G/e$. So $F$ is acyclic $\gamma^{*}$-nonzero in $G/e$ and therefore $F$ is a feasible set of $\mathcal{G}((G, \gamma)/e)$.
\end{proof}

We omit the proof of the following lemma.

\begin{lemma}\label{lem: delete vertex}
Let $(G, \gamma)$ be a $\Gamma$-labelled graph and $v$ be an isolated vertex of $G$. Then $\mathcal{G}((G, \gamma)\setminus v)=\mathcal{G}(G\setminus v, \gamma|_{V(G)-\{v\}})$.
\end{lemma}

\begin{proposition}
  \label{prop: minor graft and delta-matroid}
  Let $(G, \gamma)$ be a $\Gamma$-labelled graph and $M = \mathcal{G}(G,\gamma) \triangle X$ for some $X \subseteq E(G)$.
  \begin{enumerate}[label=\rm(\roman*)]
    \item\label{item:mi1} If $(G', \gamma')$ is a minor of $(G, \gamma)$, then $\mathcal{G}(G', \gamma')$ is a minor of $M$.
    \item\label{item:mi2} If $M'$ is a minor of $M$, then there exists a minor $(G',\gamma')$ of $(G,\gamma)$ such that $M' = \mathcal{G}(G',\gamma') \triangle X'$ for some $X' \subseteq E(G')$.
  \end{enumerate}
\end{proposition}
\begin{proof}
  We may assume that $X = \emptyset$.
  Lemmas~\ref{lem:deletion}, \ref{lem:contraction}, and \ref{lem: delete vertex} imply~\ref{item:mi1} and
  Lemmas~\ref{lem:delandcont}, \ref{lem:deletion}, \ref{lem:contraction}, and \ref{lem: delete vertex} imply~\ref{item:mi2}.
\end{proof}

\section{Maximum weight acyclic $\gamma$-nonzero set}\label{sec: applications}

In this section, we prove that one can find a maximum weight acyclic $\gamma$-nonzero set in a $\Gamma$-labelled graph $(G,\gamma)$ in polynomial time by applying the symmetric greedy algorithm on $\Gamma$-graphic delta-matroids. Let us first state the problem.

\begin{tcolorbox}
  \textsc{Maximum Weight Acyclic $\gamma$-nonzero Set} \\
  \textbf{Input:} A $\Gamma$-labelled graph $(G, \gamma)$ and a weight $w : E(G) \rightarrow \mathbb{Q}$. \\
  \textbf{Problem:} Find an acyclic $\gamma$-nonzero set $F$ in $G$ maximizing $\sum_{e \in F} w(e)$.
\end{tcolorbox}

Recall that Theorem~\ref{thm:greedy_algo} allows us to find a maximum weight feasible set in a delta-matroid by using the symmetric greedy algorithm in Algorithm~\ref{algo:greedy_algo}.
As we proved that the set of acyclic $\gamma$-nonzero sets in a $\Gamma$-labelled graph $(G,\gamma)$ forms a $\Gamma$-graphic delta-matroid in Section~\ref{sec: delta-matroids from group-labelled graphs}, we can apply Theorem~\ref{thm:greedy_algo} to solve \textsc{Maximum Weight Acyclic $\gamma$-nonzero Set}, but it requires a subroutine that decides in polynomial time whether a pair of two disjoint sets $X$ and $Y$ of $E(G)$ is separable in $\mathcal{G}(G,\gamma)$.
In the remainder of this section, we focus on developing this subroutine.

We assume that the addition of two elements of $\Gamma$ and testing whether an element of $\Gamma$ is zero can be done in time polynomial in the length of the input.

\begin{theorem}\label{thm:oracle}
  Given a $\Gamma$-labelled graph $(G,\gamma)$ and disjoint subsets $X$, $Y$ of $E(G)$, one can decide in polynomial time whether $G$ has an acyclic $\gamma$-nonzero set $F$ such that $X \subseteq F$ and $Y \cap F = \emptyset$.
\end{theorem}

To prove Theorem~\ref{thm:oracle}, we will characterize separable pairs $(X,Y)$ in $\mathcal{G}(G,\gamma)$.
Recall that, for a $\Gamma$-labelled graph $(G,\gamma)$, $\kappa(G,\gamma)$ is the number of components $C$ of $G$ such that $\gamma|_{V(C)}\equiv 0$.

\begin{lemma}
  \label{lem: kappa monotone}
  Let $\Gamma$ be an abelian group and $(G,\gamma)$ be a $\Gamma$-labelled graph.
  Then $\kappa((G,\gamma) \setminus e) \geq \kappa(G,\gamma)$ and $\kappa((G,\gamma)/e) \geq \kappa(G,\gamma)$ for every edge $e$ of $G$.
\end{lemma}
\begin{proof}
  We may assume that $G$ is connected and $\kappa(G,\gamma)=1$.
  Then $\gamma \equiv 0$ and therefore $\kappa((G,\gamma) \setminus e) \geq 1$ and $\kappa((G,\gamma)/e) = 1$.
\end{proof}

\begin{lemma}
  \label{lem: tunnel set}
  Let $\Gamma$ be an abelian group, $(G,\gamma)$ be a $\Gamma$-labelled graph, and $X$ be an acyclic set of edges of~$G$. Let $\gamma':V(G/X)\rightarrow\Gamma$ be a map such that $(G/X,\gamma')=(G,\gamma)/X$. Then the following hold.
  \begin{enumerate}[label=\rm(\arabic*)]
  \item\label{item:tunnel_cont1} If $\kappa((G,\gamma)/X)=\kappa(G,\gamma)$ and $F$ is an acyclic $\gamma'$-nonzero set in $G/X$, then $F\cup X$ is an acyclic $\gamma$-nonzero set in $G$.
  \item\label{item:tunnel_cont2} If $\kappa((G,\gamma)/X)>\kappa(G,\gamma)$, then $G$ has no acyclic $\gamma$-nonzero set containing $X$.
  \end{enumerate}
\end{lemma}
\begin{proof}
 Let us first prove~\ref{item:tunnel_cont1}. By considering each component, we may assume that $G$ is connected. Since $X$ is acyclic, $F\cup X$ is acyclic in $G$.
 
  If $\kappa((G,\gamma)/X)=\kappa(G,\gamma)=1$, then $\gamma\equiv 0$ and $F$ is the edge set of a spanning tree of $G/X$ by~\ref{item:def2}.
Hence $F\cup X$ is the edge set of a spanning tree of $G$, which implies that $F\cup X$ is acyclic $\gamma$-nonzero in~$G$.

  If $\kappa((G,\gamma)/X)=\kappa(G,\gamma)=0$, then let $H'=(V(G/X),F)$ be a subgraph of $G/X$ and $H=(V(G),F\cup X)$ be a subgraph of $G$.
  Then, for each component $C$ of $H$, there exists a component $C'$ of~$H'$ such that $C'=C/(E(C)\cap X)$. Then $\sum_{u\in V(C)}\gamma(u)=\sum_{u\in V(C')}\gamma'(u)\neq 0$ by~\ref{item:def1}.
  Hence $F\cup X$ is an acyclic $\gamma$-nonzero set in $G$ and \ref{item:tunnel_cont1} holds.
  
  Now let us prove~\ref{item:tunnel_cont2}.
  We proceed by induction on $|X|$.
  
  If $|X|=1$, then $e\in X$ is a $\gamma$-tunnel and by Lemma~\ref{lem:gtunnel}, there is no acyclic $\gamma$-nonzero set containing~$X$. So we may assume that $|X|>1$. Let $e\in X$ and $X'=X-\{e\}$.
  
  By the induction hypothesis, we may assume that $\kappa((G,\gamma)/X')=\kappa(G,\gamma)$. Let $\gamma'':V(G/X')\rightarrow\Gamma$ be a map such that $(G/X',\gamma'')=(G,\gamma)/X'$.
  Since $\kappa((G,\gamma)/X)=\kappa((G,\gamma)/X'/e)>\kappa((G,\gamma)/X')$, by the induction hypothesis, $G/X'$ has no acyclic $\gamma''$-nonzero set containing $e$. Therefore, $G$ has no acyclic $\gamma$-nonzero set containing $X$.
\end{proof}

\begin{lemma}
  \label{lem: bridge set}
  Let $\Gamma$ be an abelian group, $(G, \gamma)$ be a $\Gamma$-labelled graph, and $Y$ be a set of edges of~$G$.
  Then the following hold.
  \begin{enumerate}[label=\rm(\arabic*)]
    \item\label{item:bridge_del1} If $\kappa((G, \gamma) \setminus Y) = \kappa(G, \gamma)$ and $F$ is an acyclic $\gamma$-nonzero set in $G \setminus Y$, then $F$ is an acyclic $\gamma$-nonzero set in $G$.
    \item\label{item:bridge_del2} If $\kappa((G, \gamma) \setminus Y) > \kappa(G, \gamma)$, then $G$ has no acyclic $\gamma$-nonzero set $F$ such that $Y \cap F  = \emptyset$.
  \end{enumerate}
\end{lemma}
\begin{proof}
  Let us first prove~\ref{item:bridge_del1}.
  By considering each component, we may assume that $G$ is connected.

If $\kappa((G,\gamma)\setminus Y)=\kappa(G,\gamma)=1$, then $\gamma\equiv 0$ and the set $F$ is the edge set of a spanning tree of $G \setminus Y$ by~\ref{item:def2}.
  Then $F$ is an acyclic $\gamma$-nonzero set in $G$.
  
  If $\kappa((G,\gamma)\setminus Y)=\kappa(G,\gamma)=0$, then for each component $C$ of $G\setminus Y$, we have $\gamma|_{V(C)} \not\equiv 0$.
  Then, $\sum_{v\in V(C)} \gamma(v) \neq 0$ for each component $C$ of $(V(G), F)$.
  So $F$ is an acyclic $\gamma$-nonzero set in $G$.

  Let us show \ref{item:bridge_del2}.
  We proceed by induction on $|Y|$.
  If $|Y|=1$, then $e\in Y$ is a $\gamma$-bridge so it is done by Lemma~\ref{lem:gbridge}.
  Now we assume $|Y| \geq 2$.
  Let $e \in Y$ and $Y'=Y-\{e\}$.
  By the induction hypothesis, we may assume that $\kappa(G \setminus Y', \gamma) = \kappa(G, \gamma)$.
  Since $\kappa(G \setminus Y, \gamma)=\kappa(G \setminus Y'\setminus e, \gamma)>\kappa(G\setminus Y', \gamma)$, by the induction hypothesis, every acyclic $\gamma$-nonzero set in $G\setminus Y'$ contains $e$.
  Since every acyclic $\gamma$-nonzero set $F$ in $G$ not intersecting $Y'$ is an acyclic $\gamma$-nonzero set in $G\setminus Y'$, every acyclic $\gamma$-nonzero set in $G$ intersects $Y$.
\end{proof}

\begin{proposition}\label{prop:separating_set}
  Let $\Gamma$ be an abelian group and $(G, \gamma)$ be a $\Gamma$-labelled graph.
  Let $X$ and $Y$ be disjoint subsets of $E(G)$ such that $X$ is acyclic in $G$.
  Then $\kappa((G, \gamma) / X \setminus Y) = \kappa(G, \gamma)$ if and only if $G$ has an acyclic $\gamma$-nonzero set $F$ such that $X \subseteq F$ and $Y\cap F = \emptyset$.
\end{proposition}
\begin{proof}
Let us prove the forward direction.
  By Lemma~\ref{lem: kappa monotone}, $\kappa((G,\gamma) / X \setminus Y) = \kappa((G,\gamma) / X) = \kappa(G,\gamma)$.
  Let $\gamma':V(G/X\setminus Y)\rightarrow\Gamma$ be a map such that $(G / X \setminus Y, \gamma') = (G, \gamma) / X \setminus Y$.
  By~\ref{item:d1} of Theorem~\ref{thm:delta}, there exists an acyclic $\gamma'$-nonzero set $F'$ in $G / X \setminus Y$.
  Since $\kappa((G,\gamma) / X \setminus Y) = \kappa((G,\gamma) / X)$, $F'$ is acyclic $\gamma'$-nonzero in $G / X$ by~\ref{item:bridge_del1} of Lemma~\ref{lem: bridge set}.
  Since $\kappa((G,\gamma) / X) = \kappa(G,\gamma)$, $F := F'\cup X$ is acyclic $\gamma$-nonzero in $G$ by~\ref{item:tunnel_cont1} of Lemma~\ref{lem: tunnel set}.
  Therefore, $F$ is an acyclic $\gamma$-nonzero set in $G$ such that $X \subseteq F$ and $Y \cap F = \emptyset$.
  
  Now let us prove the backward direction.
  Let $F$ be an acyclic $\gamma$-nonzero set in $G$ such that $X \subseteq F$ and $Y \cap F=\emptyset$.
  Let $\gamma':V(G/X)\rightarrow\Gamma$ be a map such that $(G/X, \gamma') = (G, \gamma) / X$.
  Then $F-X$ is an acyclic $\gamma'$-nonzero set in $G/X$ not intersecting $Y$, so we have $\kappa((G,\gamma) / X \setminus Y) = \kappa((G,\gamma) / X)$ by~\ref{item:bridge_del2} of Lemma~\ref{lem: bridge set}.
  Since $F$ is an acyclic $\gamma$-nonzero set containing $X$ in $G$, we have $\kappa((G,\gamma) / X) = \kappa(G,\gamma)$ by~\ref{item:tunnel_cont2} of Lemma~\ref{lem: tunnel set}.
\end{proof}

\begin{proof}[Proof of Theorem~\ref{thm:oracle}]
  Given a $\Gamma$-labelled graph $(G,\gamma)$ and disjoint subsets $X$, $Y$ of $E(G)$, we can compute $\kappa((G,\gamma)/X\setminus Y)$ in polynomial time and therefore, by Proposition~\ref{prop:separating_set}, we can decide whether there exists an acyclic $\gamma$-nonzero set $F$ in $G$ such that $X \subseteq F$ and $Y \cap F = \emptyset$.
\end{proof}

Now we are ready to show Theorem~\ref{thm:main_algo}

\begin{theorem2}
  \textsc{Maximum Weight Acyclic $\gamma$-nonzero Set} is solvable in polynomial time.
\end{theorem2}

\begin{proof}%
  Let $M = \mathcal{G}(G,\gamma)$ be a $\Gamma$-graphic delta-matroid.
  The set of acyclic $\gamma$-nonzero sets in $G$ is equal to the set of feasible sets of $M$.
  By Theorem~\ref{thm:oracle}, we can decide in polynomial time whether a pair $(X,Y)$ of disjoint subsets $X$ and $Y$ of $E(G)$ is separable in $M$.
  It implies that the symmetric greedy algorithm in Algorithm~\ref{algo:greedy_algo} for $M$ and $w$ runs in polynomial time.
  By Theorem~\ref{thm:greedy_algo}, we can obtain an acyclic $\gamma$-nonzero set $F$ in $G$ maximizing $\sum_{e\in F} w(e)$.
\end{proof}

\section{Even $\Gamma$-graphic delta-matroids}\label{sec: even gamma-graphic delta-matroids}

In this section, we show that every even $\Gamma$-graphic delta-matroid is graphic.

\begin{lemma}\label{lem:even1}
  Let $(G, \gamma)$ be a $\Gamma$-labelled graph, and $\eta : V(G) \rightarrow \mathbb{Z}_2$ such that $\eta(v) = 0$ if and only if $\gamma(v) = 0$ for each $v\in V(G)$.
  If $\mathcal{G}(G,\gamma)$ is even, then, for each connected subgraph $H$ of $G$, $\sum_{u\in V(H)} \eta(u) = 0$ if and only if $\sum_{u \in V(H)} \gamma(u) = 0$.
\end{lemma}
\begin{proof}
  We proceed by induction on $|V(H)|$.
  We may assume that $|V(H)| \geq 2$.
  Then there is a vertex $v \in V(H)$ such that $H \setminus v$ is connected.
  Let $H' = H\setminus v$.
  By the induction hypothesis, $\sum_{u\in V(H')} \eta(u) = 0$ if and only if $\sum_{u \in V(H')} \gamma(u) = 0$.

  Let $(G',\gamma')$ be a $\Gamma$-labelled graph obtained from $(G,\gamma)$ by deleting all edges in $E(G) - E(H)$, contracting edges in $E(H\setminus v)$, and deleting loops and parallel edges.
  Let $e=vw$ be a unique edge of~$G'$.
  Observe that $\gamma'(v) = \gamma(v)$ and $\gamma'(w) = \sum_{u\in V(H')} \gamma(u)$.
  Let $\eta' : V(G') \rightarrow \mathbb{Z}_2$ be a map such that
  $\eta'(v)= \eta(v)$ and $\eta'(w) = \sum_{u\in V(H')} \eta(u)$.

Let us first check the backward direction. If $\sum_{u\in V(H)} \gamma(u) = \gamma'(v) + \gamma'(w) = 0$, then $\eta'(v) = \eta'(w)$ and so $\sum_{u\in V(H)} \eta(u) = \eta'(v) + \eta'(w) = 0$.
  
  Now let us prove the forward direction. Suppose that $\sum_{u\in V(H)} \eta(u) = 0$ and $\sum_{u\in V(H)} \gamma(u) \neq 0$. Then $\eta'(v) +\eta'(w)=0$ and $\gamma'(v)+\gamma'(w)\neq 0$ and therefore all of $\gamma'(v)$, $\gamma'(w)$, and $\gamma'(v) + \gamma'(w)$ are nonzero. Hence $\mathcal{G}(G',\gamma') = (\{e\}, \{\emptyset, \{e\}\})$. By Theorem~\ref{thm: excluded minor for even delta-matroids} and \ref{item:mi1} of Proposition~\ref{prop: minor graft and delta-matroid}, $\mathcal{G}(G,\gamma)$ is not even,
  which is a contradiction.
\end{proof}

\begin{proposition}\label{prop: gamma-graphic and even}
  Let $(G, \gamma)$ be a $\Gamma$-labelled graph.
  If $\mathcal{G}(G,\gamma)$ is even, then there is a map $\eta: V(G) \rightarrow \mathbb{Z}_2$ such that $\mathcal{G}(G,\gamma) = \mathcal{G}(G, \eta)$.
\end{proposition}
\begin{proof}
 
  Let $\eta: V(G) \rightarrow \mathbb{Z}_2$ is a map such that, for every $u\in V(G)$, $\eta(u) = 0$ if and only if $\gamma(u) = 0$. Let $F$ be a set of edges of $G$. Then, for each component $C$ of $(V(G),F)$, $\gamma|_{V(C)}\equiv 0$ if and only if $\eta|_{V(C)}\equiv 0$ and, by Lemma~\ref{lem:even1}, $\sum_{u\in V(C)}\gamma(u)\neq 0$ if and only if $\sum_{u\in V(C)}\eta(u)\neq 0$.
Therefore, $F$ is acyclic $\gamma$-nonzero in $G$ if and only if it is acyclic $\eta$-nonzero in $G$.
\end{proof}

We are ready to prove Theorem~\ref{thm:even}.

\begin{theorem3}
  Let $\Gamma$ be an abelian group.
  Then a $\Gamma$-graphic delta-matroid is even if and only if it is graphic.
\end{theorem3}

\begin{proof}[Proof of Theorem~\ref{thm:even}]
  Let $M$ be an even $\Gamma$-graphic delta-matroid.
  By twisting, we may assume that $M=\mathcal{G}(G,\gamma)$ for a $\Gamma$-labelled graph $(G,\gamma)$.
  By Proposition~\ref{prop: gamma-graphic and even}, $M$ is $\mathbb{Z}_2$-graphic.
  Conversely, Oum~{\cite[Theorem 5]{Oum2009circle}} proved that every graphic delta-matroid is even.
\end{proof}

\section{Representations of $\Gamma$-graphic delta-matroids}\label{sec: representations of Gamma-graphic delta-matroids}

We aim to study the condition on an abelian group $\Gamma$ and a field $\mathbb{F}$ such that every $\Gamma$-graphic delta-matroid is representable over $\mathbb{F}$.
Recall that a delta-matroid $M = (E,\mathcal{F})$ is representable over $\mathbb{F}$ if there is an $E\times E$ symmetric or skew-symmetric $A$ over $\mathbb{F}$ such that $\mathcal{F} = \{F \subseteq E: \text{$A[X]$ is nonsingular}\} \triangle X$ for some $X \subseteq E$.
If every $\Gamma$-graphic delta-matroid is representable over $\mathbb{F}$, then to prove this, we will construct symmetric matrices over $\mathbb{F}$ representing $\Gamma$-graphic delta-matroids.

For a graph $G = (V,E)$, let $\vec{G}$ be an orientation obtained from $G$ by arbitrarily assigning a direction to each edge.
Let $I_{\vec{G}} = (a_{ve})_{v\in V,\; e\in E}$ be a $V \times E$ matrix over $\mathbb{F}$ such that, for a vertex $v\in V$ and an edge $e\in E$,
\[
  a_{ve} =
  \begin{cases}
    1   & \text{if $v$ is the head of a non-loop edge $e$ in $\vec{G}$}, \\ 
    -1  & \text{if $v$ is the tail of a non-loop edge $e$ in $\vec{G}$}, \\ 
    0 & \text{otherwise}.
  \end{cases}
\]

\begin{lemma}
\label{lem:ort}
  Let $G = (V, E)$ be a graph and $\vec{G}_1$, $\vec{G}_2$ be orientations of $G$.
  If $W \subseteq V$, $F \subseteq E$, and $|W|=|F|$, then $\det(I_{\vec{G}_1}[W,F]) = \pm \det(I_{\vec{G}_2}[W,F])$.
\end{lemma}
\begin{proof}
  The matrix $I_{\vec{G}_1}$ can be obtained from $I_{\vec{G}_2}$ by multiplying $-1$ to some columns.
\end{proof}

By slightly abusing the notation, we simply write $I_G$ to denote $I_{\vec{G}}$ for some orientation $\vec{G}$ of $G$.
The following two lemmas are easy exercises.

\begin{lemma}[see Oxley~{\cite[Lemma 5.1.3]{Oxley2011matroid}}]
  \label{lem: graph and incidence matrix}
  Let $G$ be a graph and $F$ be an edge set of $G$.
  Then $F$ is acyclic if and only if column vectors of $I_G$ indexed by the elements of $F$ are linearly independent.
\end{lemma}

\begin{lemma}[see Matou\v{s}ek and Ne\v{s}et\v{r}il~{\cite[Lemma~8.5.3]{Matousek2009}}]
  \label{lem: det tree}
  Let $G = (V,E)$ be a tree.
  Then $\det(I_G[V - \{v\}, E]) = \pm 1$ for any vertex $v \in V$.
\end{lemma}

\begin{lemma}\label{lem: labelled-graft assigned by nonzeros}
  Let $\Gamma$ be an abelian group with at least one nonzero element, and $(G, \gamma)$ be a $\Gamma$-labelled graph.
  Then there is a $\Gamma$-labelled graph $(H, \eta)$ such that
  \begin{enumerate}[label=\rm(\roman*)]
    \item $\eta(v) \neq 0$ for each vertex $v\in V(H)$ and
    \item $(G, \gamma)$ is a minor of $(H, \eta)$.
  \end{enumerate}
\end{lemma}

\begin{proof}
  Let $Z(G,\gamma)$ be the set of vertices $v \in V(G)$ such that $\gamma(v) = 0$.
  We proceed by induction on~$|Z(G, \gamma)|$.

  We may assume that $v \in Z(G, \gamma)$.
  Choose a nonzero element $g \in \Gamma$.
  Let $G'$ be a graph obtained from $G$ by adding a new vertex $w$ adjacent only to $v$, and $\gamma'$ be a map from $V(G')$ to $\Gamma$ such that
  \[
    \gamma'(u) =
    \begin{cases}
      g         & \textrm{if $u = v$}, \\
      -g        & \textrm{if $u = w$}, \\
      \gamma(u) & \textrm{otherwise}.
    \end{cases}
  \]
  Then the $\Gamma$-labelled graph $(G, \gamma)$ is isomorphic to $(G', \gamma') / vw$.
  We know $|Z(G',\gamma')| = |Z(G,\gamma)|-1$.
  By the induction hypothesis, there is a $\Gamma$-labelled graph $(H,\eta)$ such that $Z(H,\eta) = \emptyset$ and $(G',\gamma')$ is a minor of $(H,\eta)$.
  The latter implies that $(G, \gamma)$ is a minor of~$(H, \eta)$.
\end{proof}

\begin{theorem}[Binet-Cauchy theorem]\label{thm: Cauchy-Binet formula}
  Let $X$ and $Y$ be finite sets.
  Let $M$ be an $X \times Y$ matrix and $N$ be a $Y \times X$ matrix with $|Y| \geq |X| = s$.
  Then
  \[
    \det (MN) = \sum_{S \in \binom{Y}{s}} \det (M[X,S]) \cdot \det (N[S, X]).
  \]
\end{theorem}

It is straightforward to prove the following lemma from the Binet-Cauchy theorem.

\begin{corollary}\label{coro: 3-term CB formula}
  Let $X$, $Y$, $Z$ be finite sets.
  Let $L$, $M$, $N$ be $X \times Y$, $Y \times Z$, $Z \times X$ matrices, respectively, with $|Y|, |Z| \geq |X| = s$.
  Then
  \[
    \det (LMN) = \sum_{S \in \binom{Y}{s}, \; T \in \binom{Z}{s}} \det(L[X, S]) \cdot \det(M[S,T]) \cdot \det(N[T,X]).
  \]
\end{corollary}

\begin{theorem4}
  Let $p$ be a prime, $k$ be a positive integer, and $\mathbb{F}$ be a field of characteristic $p$. If $[\mathbb{F}:\GF(p)]\geq k$, then every $\mathbb{Z}_p^k$-graphic delta-matroid is representable over $\mathbb{F}$.
 \end{theorem4}
\begin{proof}
  Let $M$ be a $\mathbb{Z}_p^k$-graphic delta-matroid.
  By twisting, we may assume $M = \mathcal{G}(G, \gamma)$ for a $\mathbb{Z}_p^k$-labelled graph $(G,\gamma)$. Let $V=V(G)$ and $E=E(G)$.
  We may assume that $\gamma(v) \neq 0$ for each vertex $v$ of $V$ because otherwise, by applying Lemma~\ref{lem: labelled-graft assigned by nonzeros}, we can replace $(G,\gamma)$ by a $\mathbb{Z}_p^k$-labelled graph $(H,\eta)$ such that $\eta(u)\neq 0$ for each $u\in V(H)$ and $\mathcal{G}(G,\gamma)$ is a minor of $\mathcal{G}(H,\eta)$.

  Since $[\mathbb{F}:\GF(p)]\geq k$, there exists a set $\{\alpha_1, \dots, \alpha_{k}\}$ of linearly independent vectors of $\mathbb{F}$.
  So we can take an injective group homomorphism $\phi$ from $\mathbb{Z}_{p}^{k}$ to $\mathbb{F}$ such that
  \[
    \phi(c_1, \dots, c_k) = \sum_{i=1}^{k} c_i \alpha_i.
  \]
  Let $B = (b_{vw})$ be a $V \times V$ diagonal matrix over $\mathbb{F}$ such that $b_{vv} = 1/\phi(\gamma(v))$, and let $A := I_{G}^T B I_{G}$. Then~$A$ is an $E \times E$ symmetric matrix over $\mathbb{F}$.
  Observe that $\det(B) = 1/\prod_{v\in V} \phi(\gamma(v)) \neq 0$ and $\det(B[V-\{v\}]) = \det(B)\phi(\gamma(v))$ for each $v\in V$.

  Let $F$ be a set of edges of $G$.

  \begin{claim}
  \label{claim:repre1}
    If $F$ is not acyclic in $G$, then $\det(A[F]) = 0$.
  \end{claim}
  \begin{proof}
    We have $\rnk(A[F]) \leq \rnk(I_{G}[V,F])$.
    By Lemma~\ref{lem: graph and incidence matrix}, $\rnk(I_{G}[V,F]) < |F|$, which implies that $\det(A[F]) = 0$.
  \end{proof}

  \begin{claim}
  \label{claim:repre2}
  For a connected subgraph $H$ of $G$, 
   $A[E(H)]$ is nonsingular if and only if $H$ is a tree and $\sum_{v\in V(H)} \gamma(v)\neq 0$.  
  \end{claim}
  \begin{proof}
  Let $V_{H}=V(H)$ and $E_{H}=E(H)$. 
  Let us first show that if $H$ is a tree, then $\det(A[E_{H}])=\det(B[V_{H}])\phi\left(\sum_{v\in V_{H}}\gamma(v)\right)$.
   For~$X, Y\subseteq V$ with $|X| = |Y|$, we have $\det(B[X,Y]) \neq 0$ if and only if $X=Y$ because $B$ is a diagonal matrix.    
   Then by Lemma~\ref{lem: det tree} and Corollary~\ref{coro: 3-term CB formula}, we know that
    \begin{align*}
      \det(A[E_{H}])
      &=
      \sum_{v,w \in V_{H}} \det(I_G^T[E_H, V_H-\{v\}])\det(B[V_H-\{v\}, V_H-\{w\}])\det(I_G[V_H-\{w\}, E_H]) \\
      &=
      \sum_{v\in V_H} \det(I_G[V_H-\{v\}, E_H])^2 \det(B[V_H-\{v\}]) \\
      &=
      \det(B[V_{H}]) \sum_{v\in V_H} \phi(\gamma(v))
      =
      \det(B[V_{H}]) \phi \left( \sum_{v\in V_H} \gamma(v) \right)
    \end{align*}
    
    If $A[E_H]$ is nonsingular, then $H$ is a tree by Claim~\ref{claim:repre1} and  therefore $\sum_{v\in V_H} \gamma(v)\neq 0$.
   	Conversely, if~$H$ is a tree and $\sum_{v\in V_H} \gamma(v) \neq 0$, then $\det(A[E_H]) \neq 0$ because $\phi$ is injective and $\det(B[V_{H}]) \neq 0$.
  \end{proof}

By Claim~\ref{claim:repre2}, $A[F]$ is nonsingular if and only if $F$ is acyclic $\gamma$-nonzero in $G$, which implies that $\mathcal{G}(G,\gamma)$ is representable over $\mathbb{F}$.
\end{proof}

Now we show that for some pairs of an abelian group $\Gamma$ and a finite field $\mathbb{F}$, not every $\Gamma$-graphic delta-matroid is representable over $\mathbb{F}$.
Let $R(n;m)$ be the Ramsey number that is the minimum integer $t$ such that any coloring of edges of $K_{t}$ into $m$ colors induces a monochromatic copy of~$K_{n}$.

\begin{theorem}[Ramsey~{\cite{Ramsey1929}}]
\label{thm:ramsey}
For positive integers $m$ and $n$, $R(n;m)$ is finite.
\end{theorem}

\begin{corollary}
\label{cor:ramsey2}
Let $k$ be a positive integer and $\mathbb{F}$ be a finite field of order $m$.
If $N \geq R(k;m)$, then each $N\times N$ symmetric matrix $A$ over $\mathbb{F}$ has a $k\times k$ principal submatrix $A'$ such that all non-diagonal entries are equal.
\end{corollary}

\begin{lemma}
  \label{lem: graphic not regular}
  Let $\mathbb{F}$ be a field.
  If every $\mathbb{Z}_{2}$-graphic delta-matroid is representable over $\mathbb{F}$, then the characteristic of $\mathbb{F}$ is $2$.
  \end{lemma}
  \begin{proof}
  Let $G=(V,E)$ be a graph such that $V = \{u_1, u_2, w_1, w_2, w_3\}$ and $E = \{u_iw_j : i\in \{1,2\}, \; j\in \{1,2,3\}\}$. We label edges $u_{1}w_{1}$, $u_{1}w_{2}$, $u_{1}w_{3}$ by $1$, $2$, $3$, respectively and label edges $u_{2}w_{1}$, $u_{2}w_{2}$, $u_{2}w_{3}$ by $4$, $5$, $6$, respectively.
  See Figure~\ref{fig:k23}.
  
  \begin{figure}\label{fig:k23}
    \begin{center}
      \begin{tikzpicture}
        \tikzstyle{v}=[circle,draw,fill=black!50,inner sep=0pt,minimum width=3pt]
        \foreach \i in {1,2} {
          \draw (4*\i-2,2) node [v,label=$u_{\i}$](u\i) {};
        }
        \foreach \j in {1,2,3} {
          \draw (2*\j,0) node [v,label=below:$w_{\j}$](w\j) {};
          \draw  (u1)--(w\j) node [pos=0.5,label=above left:$\j$]{};
        }
        \draw  (u2)--(w1) node [pos=0.5,label=above right:$4$]{};
        \draw  (u2)--(w2)node [pos=0.5,label=above right:$5$]{};
        \draw (u2)--(w3)node [pos=0.5,label=above:$6$]{};
      \end{tikzpicture}
    \end{center}
    \caption{A graph $G$ in the proof of Lemma~\ref{lem: graphic not regular}.}
  \end{figure}
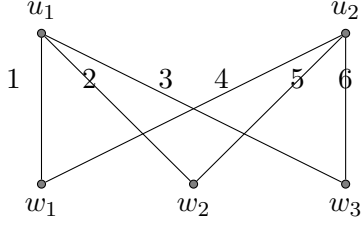

  Let $\gamma:V\rightarrow\mathbb{Z}_{2}$ be a map such that $\gamma(x)=1$ for each $x\in V$ and $M=(E,\mathcal{F})$ be a $\mathbb{Z}_{2}$-graphic delta-matroid $\mathcal{G}(G,\gamma)$. Then $M$ is even by Oum~\cite{Oum2009circle}. Since $\gamma(x)\neq 0$ for each $x\in V$, we have $\emptyset\in\mathcal{F}$.
  By Lemmas~\ref{prop: normal repre} and \ref{lem:skew}, there exists an $E\times E$ skew-symmetric matrix $A$ such that $\mathcal{F}=\mathcal{F}(A)$.
  Since the set of $2$-element acyclic $\gamma$-nonzero sets is $\{\{1,2\},\{1,3\},\{2,3\},\{1,4\},\{2,5\},\{3,6\},\{4,5\},\{4,6\},\{5,6\}\}$, by scaling rows and columns simultaneously, we may assume that 
  \[
  A=
  \begin{blockarray}{ccccccc}
  & 1 & 2 & 3 & 4 & 5 & 6 \\
  \begin{block}{c(cccccc)}
  1 & 0 & 1 & 1 & 1 & 0 & 0 \\
  2 & -1 & 0 & t_{1} & 0 & 1 & 0 \\
  3 & -1 & -t_{1} & 0 & 0 & 0 & 1 \\
  4 & -1 & 0 & 0 & 0 & t_{2} & t_{3} \\
  5 & 0 & -1 & 0 & -t_{2} & 0 & t_{4} \\
  6 & 0 & 0 & -1 & -t_{3} & -t_{4} & 0 \\
  \end{block}
  \end{blockarray}
  \] such that $t_{i}\neq 0$ for $1\leq i\leq 4$.
  We know that $\{1,2,4,5\}$, $\{1,3,4,6\}$, $\{2,3,5,6\}$ and $\{1,2,3,4,5,6\}$ are not acyclic $\gamma$-nonzero sets in $G$.
  Hence,
  $\det(A[\{1,2,4,5\}])=(t_{2}-1)^{2} = 0$,
  $\det(A[\{1,3,4,6\}])=(t_{3}-1)^{2} = 0$,
  $\det(A[\{2,3,5,6\}])=(t_{1}t_{4}-1)^{2} = 0$ and
  $\det(A[\{1,2,3,4,5,6\}])=(t_{1}t_{4} + t_2 + t_3 -1)^{2} = 0$.
  Then, we have $1=t_{1}t_{4} = -t_2 -t_3 + 1 = -1$.
  Therefore, the characteristic of $\mathbb{F}$ is $2$.
\end{proof}

\begin{theorem5}
 Let $\mathbb{F}$ be a finite field of characteristic $p$, and $\Gamma$ be an abelian group.
 If every $\Gamma$-graphic delta-matroid is representable over $\mathbb{F}$, then $\Gamma$ is an elementary abelian $p$-group.
\end{theorem5}
\begin{proof}
  If $\Gamma = \mathbb{Z}_2$, then the conclusion follows by Lemma~\ref{lem: graphic not regular}.
  So we may assume that $|\Gamma| \geq 3$.

Suppose that $\Gamma$ is not an elementary abelian $p$-group.
Then $\Gamma$ has a nonzero element $g$ whose order is not equal to $p$.
There exists a nonzero $h \in \Gamma$ such that $h \neq g$ since $|\Gamma| \geq 3$.
Let $n:=R(p+1;|\mathbb{F}|)$ and $G = (V, E)$ be a graph isomorphic to~$K_{1,n+1}$.
Let $u$ be a leaf of $G$ and $e$ be the edge incident with $u$.
Let $\gamma : V \rightarrow \Gamma$ be a map such that, for each $v\in V$, 
\[
  \gamma(v)=
  \begin{cases}
    h   & \text{if $v = u$,} \\
    g   & \text{if $v \neq u$ and $v$ is a leaf,} \\
    -g  & \text{otherwise.}
  \end{cases}
\]
Let $M = \mathcal{G}(G, \gamma)$ be a $\Gamma$-graphic delta-matroid. Then $M$ is normal because $\gamma(v)\neq 0$ for each vertex~$v$ of $G$.
By Proposition~\ref{prop: normal repre}, we may assume that $M$ is equal to $(E, \mathcal{F}(A))$ for an $E \times E$ symmetric or skew-symmetric matrix $A$ over $\mathbb{F}$.
Since $\emptyset$ and~$\{e\}$ are acyclic $\gamma$-nonzero in $G$, the delta-matroid $M$ is not even.
So $A$ is not skew-symmetric by Lemma~\ref{lem:skew}.

Since every $1$-element subset of $E - \{e\}$ is not $\gamma$-nonzero in $G$, all diagonal entries of $A[E-\{e\}]$ are zero.
By Corollary~\ref{cor:ramsey2}, there exists $F \subseteq E - \{e\}$ such that $|F| = p+1$ and all non-diagonal entries of $A[F]$ are identical.
Since the characteristic of the field~$\mathbb{F}$ is $p$, the submatrix $A[F]$ is singular.
Therefore, $F$ is not acyclic $\gamma$-nonzero in $G$.

For each component $C$ of a subgraph $(V(G),F)$, we have $\sum_{v \in V(C)} \gamma(v)\in\{h, g, pg\}$.
This implies that $F$ is acyclic $\gamma$-nonzero in $G$, which is a contradiction.
\end{proof}

\providecommand{\bysame}{\leavevmode\hbox to3em{\hrulefill}\thinspace}
\providecommand{\MR}{\relax\ifhmode\unskip\space\fi MR }
\providecommand{\MRhref}[2]{%
  \href{http://www.ams.org/mathscinet-getitem?mr=#1}{#2}
}
\providecommand{\href}[2]{#2}


\end{document}